\documentclass[reqno, 11pt, a4paper]{amsart}

\usepackage{latexsym,ifthen,xspace}
\usepackage{amsmath,amssymb,amsthm}
\usepackage{bbm}
\usepackage{enumerate, calc}
\usepackage[colorlinks=true, pdfstartview=FitV, linkcolor=blue,
citecolor=blue, urlcolor=blue, pagebackref=false]{hyperref}
\usepackage[text={33pc,605pt},centering]{geometry}    
\usepackage{graphicx}
\usepackage{color}

\usepackage{mathtools}       

\usepackage{tikz}
\usepackage{pgfplots}
\usepackage{filecontents}

\usetikzlibrary{positioning, shapes.geometric,shapes.arrows}
\usepackage{subcaption}

\newtheorem{theorem}{Theorem}[section]
\newtheorem{lemma}[theorem]{Lemma}

\theoremstyle{definition}
\newtheorem{definition}[theorem]{Definition}
\newtheorem*{definition*}{Definition}

\theoremstyle{remark}
\newtheorem{remark}[theorem]{Remark}

\numberwithin{equation}{section}

\DeclareMathAlphabet{\mathsl}{OT1}{cmss}{m}{sl}
\SetMathAlphabet{\mathsl}{bold}{OT1}{cmss}{bx}{sl}

\newcommand{\cA}{\ensuremath{\mathcal A}}
\newcommand{\cB}{\ensuremath{\mathcal B}}

\newcommand{\cL}{\ensuremath{\mathcal L}}

\newcommand{\cN}{\ensuremath{\mathcal N}}

\newcommand{\cV}{\ensuremath{\mathcal V}}

\newcommand\cX{\ensuremath{\mathcal X}}
\newcommand{\cY}{\ensuremath{\mathcal Y}}
\newcommand{\cZ}{\ensuremath{\mathcal Z}}
\newcommand{\bbN}{\ensuremath{\mathbb N}} 
 
\newcommand{\bbP}{\ensuremath{\mathbb P}}

\newcommand{\bbZ}{\ensuremath{\mathbb Z}} 
\newcommand{\ldef}{\ensuremath{\mathrel{\mathop:}=}}

\newcommand{\indicator}{\ensuremath{\mathbbm{1}}}

\begin{document}

\title[Contact Processes with switching]{The Contact Processes with switching}


\author{Jochen Blath}
\address{Goethe-Universität Frankfurt am Main}
\curraddr{Institut für Mathematik, Robert-Mayer Strasse 10, 60629 Frankfurt}
\email{blath@math.uni-frankfurt.de}
\thanks{}


\author{Felix Hermann}
\address{Goethe-Universität Frankfurt am Main}
\curraddr{Institut für Mathematik, Robert-Mayer Strasse 10, 60629 Frankfurt}
\email{hermann@math.uni-frankfurt.de}
\thanks{}

\author{Michel Reitmeier}
\address{}
\curraddr{}
\email{}
\thanks{}



\date{\today}


\begin{abstract}
In this paper, we introduce a type switching mechanism for the Contact Process on the lattice $\mathbb{Z}^d$. That is, we allow the individual particles/sites to switch between two (or more) types independently of one another, and the different types may exhibit specific infection and recovery dynamics. 
 Such type switches can eg.\ be motivated from biology, where `phenotypic switching' is common among micro-organisms.  Our framework includes as special cases systems with switches between `active' and `dormant' states (the Contact Process with dormancy, CPD), and the Contact Process in a randomly evolving environment (CPREE) introduced by Broman (2007). The  `standard' multi-type Contact Process (without type-switching) can also be recovered as a limiting case.

After constructing the process from a graphical representation, we establish several basic properties that are mostly analogous to the classical Contact Process.  We then  provide couplings between several variants of the system, which in particular yield the existence of a phase transition.
Further, we investigate the effect of the switching parameters on the critical value of the system by providing rigorous bounds obtained from the coupling arguments as well as numerical and heuristic results.  Finally,  we investigate scaling limits for the process as the switching parameters tend to 0 (slow switching regime)  resp.\ $\infty$  (fast switching regime). We conclude with a brief discussion of further model variants and questions for future research. 
\end{abstract}

\maketitle

\section{The Contact Process with Switching}
\label{sec:intro}

\subsection{Model description}

We introduce the Contact Process with switching (CPS) as a continuous-time Markov process on the grid $S\ldef\mathbb{Z}^d$ as follows. 
At each time $t \ge 0$, the state  of the process is a function $\xi_t : \mathbb{Z}^d \to F$, where 
\begin{gather*}
F\ldef\{(0,a),(0,d),(1,a),(1,d)\}
\end{gather*} 
describes the possible states of the individuals at each grid point. (We call this the \emph{map-based representation}.)
The first component of an element from $F$ indicates whether the corresponding site (or ``particle'') is infected (`1') or healthy (`0'), and the second component refers to the type of the particle(say, `$a$' or `$d$'). The dynamics is  as follows:
Independently of all other particles an infected particle of type $a$ (resp.\ $d$) recovers with rate $\delta_a$ (resp.\ $\delta_d$), and any infected particle of type $\tau_1\in\{a,d\}$ infects a healthy neighbour of type $\tau_2\in\{a,d\}$ at rate $\lambda_{\tau_1\tau_2}\geq0$.
Finally, healthy type $a$ particles switch into type $d$ at rate $\sigma_0$, and from $d$ to $a$ at rate $\rho_0$
(with similar rates $\sigma_1$ and $\rho_1$ for the infected individuals). See Figure \ref{figure:CPS} for a visualization of these transition rates at a given site $x$, where we define 
\begin{gather*}
n_{1,a}(x,\xi)\ldef |\{z\in\mathcal{N}:\xi(x+z)=(1,a)\}|
\end{gather*}
as the number of infected type $a$ individuals in the (finite) neighbourhood of $x$ given by $\mathcal{N} \subset \mathbb{Z}^d$, and correspondingly,
\begin{gather*}
n_{1,d}(x,\xi)\ldef |\{z\in\mathcal{N}:\xi(x+z)=(1,d)\}|
\end{gather*}
for the number of infected  neighbours of $x$ of type  $d$.

For simplicity, we restrict ourselves here to finite neighbourhoods $\mathcal{N}$, so that the CPS $\{\xi_t\}$ can readily be constructed e.g.\ from a graphical representation (see Section \ref{ssn:graphical_construction} for details).  Further, we  assume $\sigma_0=\sigma_1=\sigma$ and $\rho_0=\rho_1=\rho$ unless stated otherwise.

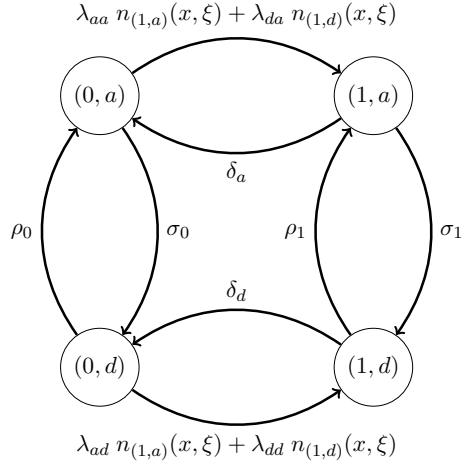
\begin{figure}[h]
	\centering
	\scalebox{.8}{
		\begin{tikzpicture}
		\node (A) at (0,0) [circle,draw] {$(0,a)$};
		\node (B) at (4.5,0) [circle,draw]{$(1,a)$};
		\node (C) at (0,-4.5) [circle,draw]{$(0,d)$};
		\node (D) at (4.5,-4.5) [circle,draw]{$(1,d)$};
		
		\draw[->, very thick] (A) to[bend left=35]node[above] {$\lambda_{aa}~n_{(1,a)}(x,\xi)+\lambda_{da}~n_{(1,d)}(x,\xi)$} (B);
		\draw[->, very thick] (C) to[bend right=35]node[below] {$\lambda_{ad}~n_{(1,a)}(x,\xi)+\lambda_{dd}~n_{(1,d)}(x,\xi)$} (D);
		\draw[<-, very thick] (A) to[bend right=35]node[below] {$\delta_a$} (B);
		\draw[<-, very thick] (C) to[bend left=35]node[above] {$\delta_d$} (D);
		\draw[->, very thick] (A) to[bend left=35]node[right] {$\sigma_0$} (C);
		\draw[<-, very thick] (A) to[bend right=35]node[left] {$\rho_0$} (C);
		\draw[->, very thick] (B) to[bend left=35]node[right] {$\sigma_1$} (D);
		\draw[<-, very thick] (B) to[bend right=35]node[left] {$\rho_1$} (D);

		\end{tikzpicture}}
	\caption{Flip rates of the Contact Process with switching (CPS).}
	\label{figure:CPS}
\end{figure}

The above type switches can be motivated for example by {\em phenotypic switches}
fostering persistence in bacterial communities, cf.\ Balaban et al.\ \cite{Balaban2004}.  The increased phenotypic diversity obtained from different switching mechanisms can often be understood as bet-hedging strategy in fluctuating environments  and has recently been investigated eg.\ in branching process based models \cite{DMB11, BHS21}.
The special case of  {\em Microbial dormancy} \cite{LHWB21} may be seen as drastic form of such a type switch,  where individuals enter inactive states with vanishing metabolic activity. This seems e.g.\ to be an efficient strategy to curb virus epidemics,  cf.\ e.g.\ \cite{JF19, BT21}, and  this  partially motivated the present study.

\subsection{Special cases}
Obviously,  a {\em first special case} of the CPS model is given by the classical Contact Process, denoted by CP,  with parameters $\lambda$ and $\delta$ chosen as
$$
\lambda_{aa}=\lambda_{ad}=\lambda_{dd}=\lambda_{da}=  \lambda \quad \mbox{ and } \quad  \delta_a= \delta_d=\delta.
$$ 
A {\em second special case} can be obtained from choosing 
$$
\lambda_{aa}=\lambda,\quad \mbox{ and }\quad  \lambda_{ad}=\lambda_{da}= \lambda_{dd}= 0.
$$
This process, in which the type $d$ particles neither infect neighbours nor get infected, will be called the Contact Process with (host) dormancy (CPD).
This special case served as motivation for  the notation $a$ (for `active state') and $d$ (for `dormant state'). See Figure \ref{figure:CPD-CPREE} (left) for a visualization of the corresponding flip rates. 
 
 Intuitively, the incorporation of dormancy should lead to an increase of the critical infection rate of the process (compared to the classical Contact Process), and this is indeed the case as we will see below. However, the strength of the observed effect will depend on the particular model parameters, and here further natural parameter choices arise.
 For example, if one interprets (microbial) host dormancy as a state of vanishing metabolic activity, then this suggests that the recovery rate $\delta_d$ should be close to zero (the host does not recover from an infection while dormant). In contrast,  if one considers (human) dormancy from an epidemiological point of view as periods of reduced social contacts, then during such `dormant states', recovery should still be possible (while the ability to infect neighbours should be 0). 
Clearly, the second strategy seems more efficient for the hosts since it allows recovery during dormancy. This intuition will be confirmed later on (cf.\ Remark~\ref{remark:distancing-dormancy}).

 Naturally, one could also consider dormancy on the level of infectors instead of hosts. For example, when the activity of the infector can be linked to a specific state of the host, such as the reactivation of Herpes viruses which can be triggered by periods of stress of the host (cf.\ \cite{KME18}), then the CPS model could still be applicable. However, if the dormancy state of an infector switches independently of the state of the host, this will lead  to yet another model class outside the scope of the present paper. A possible corresponding `CPID' model will briefly be discussed in Section~\ref{sec:CPID}.

\medskip

A {\em third special case} of our model is given by the Contact Process in a randomly evolving environment (CPREE) investigated by Broman \cite{Broman2007}. Here, only the recovery rates depend on the type of the particle in question.
Denoting by $n_{(1)} (x, \xi)$ the number of infected individuals in the $\mathcal{N}$-neighbourhood of $x$ in $\xi$, its flip rates are given by Figure \ref{figure:CPD-CPREE} (right).

In this model, the states $a$ and $d$ can be interpreted as the two states of a random environment which evolves independently for each individual whose susceptibility for infection and ability for recovery is governed by the state of the environment.

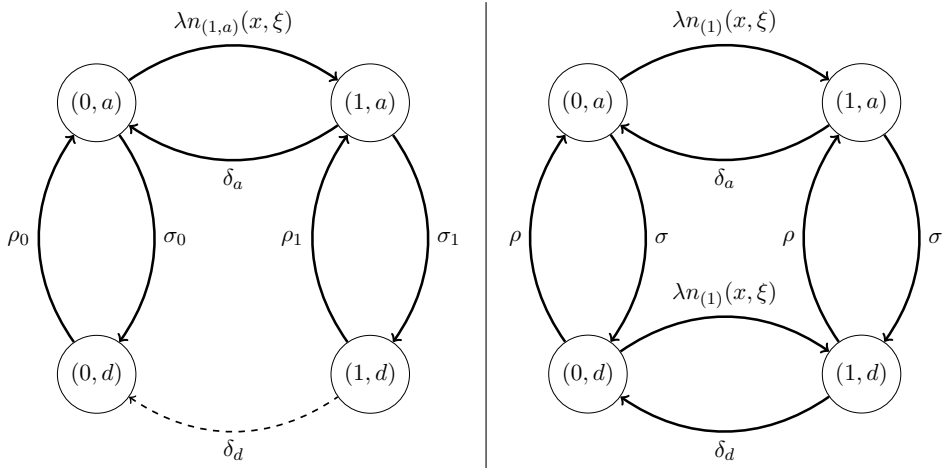
\begin{figure}[h]
	\centering
	\begin{minipage}{0.45\textwidth}
	  \scalebox{0.8}{
		\begin{tikzpicture}
		\node (A) at (0,0) [circle,draw] {$(0,a)$};
		\node (B) at (4.5,0) [circle,draw]{$(1,a)$};
		\node (C) at (0,-4.5) [circle,draw]{$(0,d)$};
		\node (D) at (4.5,-4.5) [circle,draw]{$(1,d)$};
		
		\draw[->, very thick] (A) to[bend left=35]node[above] {$\lambda n_{(1,a)}(x,\xi)$} (B);
		\draw[<-, very thick] (A) to[bend right=35]node[below] {$\delta_a$} (B);
		\draw[<-, thick, dashed] (C) to[bend right=35]node[below] {$\delta_d$} (D);
		\draw[->, very thick] (A) to[bend left=35]node[right] {$\sigma_0$} (C);
		\draw[<-, very thick] (A) to[bend right=35]node[left] {$\rho_0$} (C);
		\draw[->, very thick] (B) to[bend left=35]node[right] {$\sigma_1$} (D);
		\draw[<-, very thick] (B) to[bend right=35]node[left] {$\rho_1$} (D);

		\end{tikzpicture}}
	\end{minipage}
	\ \vrule\ 
	\begin{minipage}{.45\textwidth}
	  \scalebox{.8}{
		\begin{tikzpicture}
		\node (A) at (0,0) [circle,draw] {$(0,a)$};
		\node (B) at (4.5,0) [circle,draw]{$(1,a)$};
		\node (C) at (0,-4.5) [circle,draw]{$(0,d)$};
		\node (D) at (4.5,-4.5) [circle,draw]{$(1,d)$};
		
		\draw[->, very thick] (A) to[bend left=35]node[above] {$\lambda n_{(1)}(x,\xi)$} (B);
		\draw[->, very thick] (C) to[bend left=35]node[above] {$\lambda n_{(1)}(x,\xi)$} (D);
		\draw[<-, very thick] (A) to[bend right=35]node[below] {$\delta_a$} (B);
		\draw[<-, very thick] (C) to[bend right=35]node[below] {$\delta_d$} (D);
		\draw[->, very thick] (A) to[bend left=35]node[right] {$\sigma$} (C);
		\draw[<-, very thick] (A) to[bend right=35]node[left] {$\rho$} (C);
		\draw[->, very thick] (B) to[bend left=35]node[right] {$\sigma$} (D);
		\draw[<-, very thick] (B) to[bend right=35]node[left] {$\rho$} (D);
		
		\end{tikzpicture}}
	\end{minipage}
	\caption{Left: Flip rates of the Contact Process with dormancy (CPD). Here, $\delta_d$ can be either positive or zero.  Right: The Contact Process in a randomly evolving environment (CPREE).}
	\label{figure:CPD-CPREE}
\end{figure}

\section{Basic properties of the Contact Process with switching}

\subsection{Graphical construction and self-duality} 
\label{ssn:graphical_construction}

We introduce a graphical construction of the CPS, which is the key to establish several couplings within our model and the CP. Based on this construction, we also find that  the  CPS has a dual process within the same class of processes and even obtain an exact  \textit{self-duality} in the case $\lambda_{ad}=\lambda_{da}$.\\
The construction is a simple extension of the classical graphical representation  for the basic CP given by Harris \cite{Harris1978} and relies on families of independent Poisson processes driving infection and recovery events, combined with an idea from the construction of  the CPREE \cite{SW08} to determine the activity states of the particles.\\
Recalling the \emph{map}-based representation  $\{\xi_t\}$ form the introduction, note that here it will be convenient to work with the equivalent \emph{set}-based representation involving $$\cA_t:=\{x\in S\mid \xi_t(x)_2=a\},$$ the set of all active sites at time $t$, and $$\cX_t:=\{x\in S\mid \xi_t(x)_1=1\},$$ the set of all infected sites at time $t$,   where $\xi_t(x)_i$ denotes the $i$th component of $\xi_t(x)$.

To carry out the construction,   let the following families of independent Poisson Point Processes (PPPs) be defined on a common probability space $\big(\Omega,\mathcal{F},\mathbb{P}\big)$:
\begin{itemize}
	\item  The (potential) {\em recovery events}, given by PPPs $U^{a,x}=\{U_n^{a,x},n\geq1\}$, $U^{d,x}=\{U_n^{d,x},n\geq1\}$, 
	for every $x\in S$ with rates $\delta_a$ and $\delta_d$, respectively.	
	\item  The (potential) {\em type switching events}, given by the set of PPPs
	$V^{x,a\to d}$ and $V^{x,d\to a}$ for every $x\in S$ with rates $\sigma$ and $\rho$, respectively.
	\item  The potential {\em infection events}, given by PPPs $T^{x\to y,aa}$, $T^{x\to y,ad}$, $T^{x\to y,da}$ and $T^{x\to y,dd}$ for every pair of neighbours $(x,y)\in S^2$ with rates $\lambda_{aa}$, $\lambda_{ad}$, $\lambda_{da}$ and $\lambda_{dd}$, respectively. 
\end{itemize}
For convenience we denote these collections of PPPs by  $\mathcal{U}$, $\mathcal{V}$ and $\mathcal{T}$, respectively.

\smallskip

On the space time grid $S\times[0,\infty)$ we draw $\delta_a$'s, $\delta_d$'s, $\sigma$'s or $\rho$'s for every arrival time in $U^{a,x}$, $U^{d,x}$, $V^{x,a\to d}$ or $V^{x,d\to a}$, respectively.
A $\delta_a$ (or $\delta_d$) at $(x,t)$ indicates that the particle $x$ recovers from infection if it is infected and active (or dormant) at time $t$. The $\sigma$'s and $\rho$'s mark changes from active to dormant and vice versa. Moreover, we draw arrows of type $\tau\in\{aa,ad,da,dd\}$ from $(x,t)$ to $(y,t)$, whenever $t\in T^{x\to y,\tau}$. In particular, an arrow of type $aa$ from $(x,t)$ to $(y,t)$ signifies that $y$ gets infected by $x$ at time $t$ if both particles are active and $x$ is infected. The graphical construction is illustrated in Figure~\ref{figure:graphical-constr+dual}.

Now, we can construct the process of active particles, $(\cA_t)_{t\geq0}$, by letting
$\cV_{x,t}=(V^{x,d\to a}\cup V^{x,a\to d})\cap[0,t]$, the switching times of $x$ until time $t$, and writing
\begin{align*}
  \cA_t
   \ldef \Big\{x\in \bbZ^d \Big| \cV_{x,t}=\emptyset \wedge x\in\cA_0
           \ \text{ or }\ \max(\cV_{x,t}) \in V^{x,d\to a}\Big\},
\end{align*}
such that a particle is active, iff either it was active in the beginning and no switching occurred or the last switching time before $t$ is of type $d\to a$. Finally, this can be used to define the notion of \emph{infection paths} similarly to \cite[Definition 2.3]{SW08}. Note that we desist from their notion of \emph{active} paths to avoid confusion.

For $x,y\in S$ we say there \textit{exists an infection path from} $(x,0)$ \textit{to} $(y,t)$ if there exists a time sequence $s_0=0<s_1<s_2<\dots<s_{\ell+1}=t$ with $\ell\in\bbN$ and a sequence of particles $x_0=x,x_1,x_2,\dots,x_\ell=y$ such that:
	\begin{itemize}
		\item[(i)] For $i\in[\ell]$ there exists at time $s_i$ an arrow from $x_{i-1}$ to  $x_i$ and the activity states of  $x_{i-1}$ and $x_i$ at time $s_i$ \textit{match} with the type of the arrow, i.e.\ if both particles are for example active at time $s_i$ then the arrow has to be of type $aa$.
		\item[(ii)] For any $\delta_a$ at $(\tilde{x},t)\in\{x_i\}\times(s_i,s_{i+1})$ it holds $\tilde{x}\notin \mathcal{A}_{t}$.
		\item[(iii)] For any $\delta_d$ at $(\tilde{x},t)\in\{x_i\}\times(s_i,s_{i+1})$ it holds $\tilde{x}\in \mathcal{A}_{t}$.
	\end{itemize}
Consequently, for $t>0$ we define
\begin{gather*}
\cX_t\ldef\{y\in\bbZ^d: \text{for } x \in \cX_0 \text{ exists an infection path from } (x,0) \text{ to }(y,t)\}
\end{gather*}
the set of infected particles at time $t$. The pair $(\cX_t,\cA_t)_t$ now gives an explicit construction of the CPS via the \emph{set}-based representation.

In the same way as in \cite[(2.1) Theorem]{D95} we can find $t_0>0$ such that the interaction graph splits almost surely into clusters of finite size, making sure that the CPS is well-defined.

\begin{remark}
  \begin{enumerate}
    \item
      To specify initial configurations of infected particles $I\subset\bbZ^d$
      and of active particles $A\subset\bbZ^d$, i.e.\  choosing 
      $\cX_0=I$ and $\cA_0=A$, we  use the notation  $\cX_t^{I,A}$ and $\cA_t^{I,A}$ as well
      as $\xi_t^{I,A}$.
    \item
      Strictly speaking, the above defined process is left- but not right-continuous (cf. \cite[p.127]{D95}). Nevertheless, the left- and right-continuous versions are almost surely equal at any fixed time $t$ because the process has only countably many jumps. Hence, if we insist on càdlàg paths, we simply can redefine our process as the corresponding right-continuous version.
  \end{enumerate}
\end{remark}

From this graphical construction, by the usual thinning and coupling arguments, the following basic properties  can easily be established: 

\begin{theorem}
\label{thm:mon-add-att}
  Let $(\cX_t,\cA_t)_{t\geq0}$ and $(\tilde\cX_t,\cA_t)_{t\geq0}$ two CPS with coupled activity states.
  \begin{enumerate}
    \item If $\cX_0\subseteq\tilde\cX_0$ and the respective parameters satisfy $\tilde\lambda_\tau\geq\lambda_\tau\geq0$ for each $\tau\in\{aa,ad,da,dd\}$ and $0\leq\tilde\delta_\bullet\leq\delta_\bullet$ for every $\bullet\in\{a,d\}$, then there exists a coupling of $(\cX_t)$ and $(\tilde\cX_t)$ such that $$\cX_t\subseteq\tilde\cX_t, \qquad \mbox{for all $t \ge 0.$ \quad (monotonicity)}$$ 
    \item If $I_1,I_2\subseteq S$, then 
    $$\cX_t^{I_1,\cA_0}\cup\cX_t^{I_2,\cA_0}=\cX_t^{I_1\cup I_2,\cA_0}, \qquad \mbox{for all $t \ge 0.$ \quad (additivity)}
    $$
    \item If $I\subseteq J\subseteq S$, then 
     $$\cX_t^{I,\cA_0}\subseteq\cX_t^{J,\cA_0} \qquad \mbox{for all $t \ge 0.$ \quad (attractivity)}
     $$
  \end{enumerate}
\end{theorem}

As for the basic CP and the CPREE we obtain the dual process starting at a fixed time $t$ and letting the process run backwards in time, denoting $\hat\cX_s^t=\cX_{t-s}$ and $\hat{\cA}_s^t=\cA_{t-s}$ for $0\leq s\leq t$.  By reverting all arrows in the graphical construction of the forward-time version, one verifies that $(\hat\cX_s^t,\hat{\cA}_s^t)_s$ follows almost the same dynamics as $(\cX_s,\cA_s)_s$, except for the roles of $ad$-arrows and $da$-arrows being switched. Hence, it holds:

\begin{theorem}[Duality]\label{thm:duality}
  The dual process of a CPS $(\cX_t,\cA_t)_t$ is also distributed as a CPS with the same parameters as $(\cX_t,\cA_t)_t$ except for $\lambda_{ad}$ and $\lambda_{da}$ being swapped. Consequently, if it holds  that  $\lambda_{ad}=\lambda_{da}$,  then  $(\cX_t,\cA_t)$ is \emph{self-dual}.
\end{theorem}

 In particular, all special cases mentioned in the introduction are self-dual. The construction of the dual process is also illustrated  in Figure~\ref{figure:graphical-constr+dual}. Further, from the picture it follows immediately that the so called \textit{duality relation} holds, namely
\begin{gather*}
  \{\cX_t^{I,A} \cap J \neq \emptyset\}
    = \{\cX_s^{I,A} \cap \hat\cX_s^{t,J,\cA_t} \neq \emptyset\}
    = \{I \cap \hat\cX_t^{t,J,\cA_t}\neq \emptyset\}
    \quad\text{for all } 0\leq s\leq t.
\end{gather*}
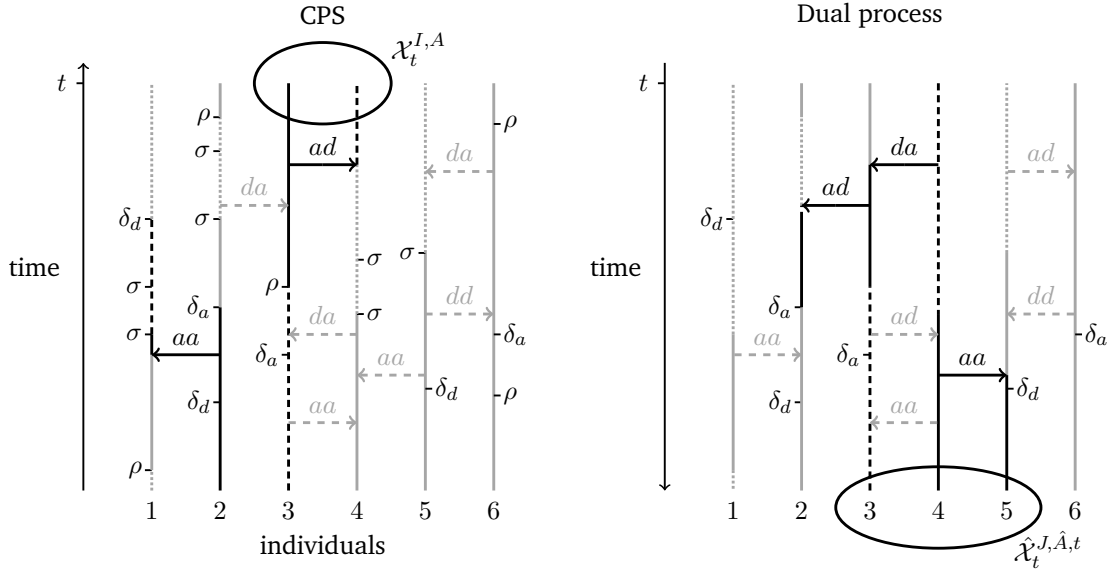
\begin{figure}[h]
	\hspace*{-1cm}%
	\scalebox{0.9}{
		\begin{tikzpicture}[darkstyle/.style={circle,inner sep=0pt}]
		\node () at (3.5,7){CPS};
		
		\draw[->,thick](0,0)--(0,6.3)node[left,xshift=-0.2cm,yshift=-3cm ] {time};
		\draw[-,thick](-0.1,6)--(0,6)node[left,xshift=-0.1cm]{$t$};
		\node () at (1,-0.3){$1$};
		\node () at (2,-0.3){$2$};
		\node () at (3,-0.3){$3$};
		\node () at (4,-0.3){$4$};
		\node () at (5,-0.3){$5$};
		\node () at (6,-0.3){$6$};
		\node () at (3.5,-0.8){individuals};
		
		\draw[-,very thick,color=gray!70!white,densely dotted] (1,0) to (1,0.3);
		\draw[-,thick](0.9,0.3)--(1,0.3)node[left]{$\rho$};
		\draw[-,very thick,color=gray!70!white] (1,0.3) to (1,2);
		\draw[-,very thick,color=black] (1,2) to (1,2.3);
		\draw[-,thick](0.9,2.3)--(1,2.3)node[left]{$\sigma$};
		\draw[-,thick](0.9,3)--(1,3)node[left]{$\sigma$};
		\draw[-,thick](0.9,4)--(1,4)node[left]{$\delta_d$};
		\draw[-,very thick,color=gray!70!white,densely dotted] (1,4) to (1,6);
		\draw[-,very thick,color=black,densely dashed] (1,2.3) to (1,4);
		\draw[-,thick](1.9,1.3)--(2,1.3)node[left]{$\delta_d$};
		\draw[-,very thick,color=black] (2,0) to (2,2.7);
		\draw[-,thick](1.9,2.7)--(2,2.7)node[left]{$\delta_a$};
		\draw[-,very thick,color=gray!70!white] (2,2.7) to (2,4);
		\draw[-,thick](1.9,4)--(2,4)node[left]{$\sigma$};
		\draw[-,thick](1.9,5)--(2,5)node[left]{$\sigma$};
		\draw[-,very thick,color=gray!70!white,densely dotted] (2,4) to (2,5.5);
		\draw[-,thick](1.9,5.5)--(2,5.5)node[left]{$\rho$};
		\draw[-,very thick,color=gray!70!white] (2,5.5) to (2,6);
		\draw[-,very thick,color=black,densely dashed] (3,0) to (3,3);
		\draw[-,thick](2.9,2)--(3,2)node[left]{$\delta_a$};
		\draw[-,thick](2.9,3)--(3,3)node[left]{$\rho$};
		\draw[-,very thick,color=black] (3,3) to (3,6);
		\draw[-,very thick,color=gray!70!white] (4,0) to (4,2.6);
		\draw[-,thick](4.1,2.6)--(4,2.6)node[right]{$\sigma$};
		\draw[-,thick](4.1,3.4)--(4,3.4)node[right]{$\sigma$};
		\draw[-,very thick,color=gray!70!white,densely dotted] (4,2.6) to (4,4.8);
		\draw[-,very thick,color=black,densely dashed] (4,4.8) to (4,6);
		\draw[-,very thick,color=gray!70!white] (5,0) to (5,3.5);
		\draw[-,thick](5.1,1.5)--(5,1.5)node[right]{$\delta_d$};
		\draw[-,thick](4.9,3.5)--(5,3.5)node[left]{$\sigma$};
		\draw[-,very thick,color=gray!70!white,densely dotted] (5,3.5) to (5,6);
		\draw[-,very thick,color=gray!70!white] (6,0) to (6,6);
		\draw[-,thick](6.1,1.4)--(6,1.4)node[right]{$\rho$};
		\draw[-,thick](6.1,5.4)--(6,5.4)node[right]{$\rho$};
		\draw[-,thick](6.1,2.3)--(6,2.3)node[right]{$\delta_a$};
		
		\draw[->,very thick,color=black] (2,2) to (1,2)node[above,xshift=0.5cm]{$aa$};
		\draw[-,very thick,color=black] (3,4.8) to (3.5,4.8)node[above]{$ad$};
		\draw[->,very thick,color=black] (3.5,4.8) to (4,4.8);
		\draw[->,dashed,very thick,color=gray!70!white] (3,1) to (4,1)node[above,xshift=-0.5cm]{$aa$};
		
		\draw[->,dashed,very thick,color=gray!70!white] (2,4.2) to (3,4.2)node[above,xshift=-0.5cm]{$da$};
		
		\draw[->,dashed,very thick,color=gray!70!white] (4,2.3) to (3,2.3)node[above,xshift=0.5cm]{$da$};
		\draw[->,dashed,very thick,color=gray!70!white] (5,1.7) to (4,1.7)node[above,xshift=0.5cm]{$aa$};
		\draw[->,dashed,very thick,color=gray!70!white] (5,2.6) to (6,2.6)node[above,xshift=-0.5cm]{$dd$};
		\draw[->,dashed,very thick,color=gray!70!white] (6,4.7) to (5,4.7)node[above,xshift=0.5cm]{$da$};
		
		\node[ellipse,
		draw = black, very thick,text=black,
		minimum width = 2cm, 
		minimum height = 1.2cm,label={[xshift=1.4cm,yshift=-0.5cm ]$\cX_t^{I,A}$}] (e) at (3.5,6) {};

		\begin{scope}[shift={(8.5,0)}]
		\node () at (3,7){Dual process};
		
		\draw[<-,thick](0,0)--(0,6.3)node[left,xshift=-0.2cm,yshift=-3cm ] {time};
		\draw[-,thick](-0.1,6)--(0,6)node[left,xshift=-0.1cm]{$t$};
		\node () at (1,-0.3){$1$};
		\node () at (2,-0.3){$2$};
		\node () at (3,-0.3){$3$};
		\node () at (4,-0.3){$4$};
		\node () at (5,-0.3){$5$};
		\node () at (6,-0.3){$6$};
		
		\draw[-,thick](0.9,4)--(1,4)node[left]{$\delta_d$};
		\draw[-,very thick,color=gray!70!white,densely dotted] (1,2.3) to (1,6);
		\draw[-,very thick,color=gray!70!white] (1,2.3) to (1,0.3);
		\draw[-,very thick,color=gray!70!white,densely dotted] (1,0.3) to (1,0);
		
		\draw[-,thick](1.9,1.3)--(2,1.3)node[left]{$\delta_d$};
		\draw[-,thick](1.9,2.7)--(2,2.7)node[left]{$\delta_a$};
		\draw[-,very thick,color=gray!70!white] (2,0) to (2,2.7);
		\draw[-,very thick,color=black] (2,2.7) to (2,4);
		\draw[-,very thick,color=black,densely dashed] (2,4) to (2,4.2);
		\draw[-,very thick,color=gray!70!white,densely dotted] (2,4.2) to (2,5.5);
		\draw[-,very thick,color=gray!70!white] (2,5.5) to (2,6);
		
		\draw[-,very thick,color=black,densely dashed] (3,0) to (3,3);
		\draw[-,thick](2.9,2)--(3,2)node[left]{$\delta_a$};
		\draw[-,very thick,color=gray!70!white] (3,6) to (3,4.8);
		\draw[-,very thick,color=black] (3,4.8) to (3,3);
		
		\draw[-,very thick,color=black,densely dashed] (4,6) to (4,2.6);
		\draw[-,very thick,color=black] (4,2.6) to (4,0);

		\draw[-,thick](5.1,1.5)--(5,1.5)node[right]{$\delta_d$};
		\draw[-,very thick,color=gray!70!white,densely dotted] (5,3.5) to (5,6);
		\draw[-,very thick,color=gray!70!white](5,3.5) to (5,1.7);
		\draw[-,very thick,color=black] (5,1.7) to (5,0);
		
		\draw[-,very thick,color=gray!70!white] (6,0) to (6,6);
		\draw[-,thick](6.1,2.3)--(6,2.3)node[right]{$\delta_a$};
		
		\draw[<-,very thick,dashed,color=gray!70!white] (2,2) to (1,2)node[above,xshift=0.5cm]{$aa$};
		\draw[<-,very thick,color=black] (3,4.8) to (3.5,4.8)node[above]{$da$};
		\draw[-,very thick,color=black] (3.5,4.8) to (4,4.8);
		\draw[<-,dashed,very thick,color=gray!70!white] (3,1) to (4,1)node[above,xshift=-0.5cm]{$aa$};
		
		\draw[<-,very thick,color=black] (2,4.2) to (2.5,4.2);
		\draw[-,color=black,very thick] (2.5,4.2) to (3,4.2)node[above,xshift=-0.5cm]{$ad$};
		
		\draw[<-,dashed,very thick,color=gray!70!white] (4,2.3) to (3,2.3)node[above,xshift=0.5cm]{$ad$};
		\draw[<-,color=black,very thick] (5,1.7) to (4,1.7)node[above,xshift=0.5cm]{$aa$};
		\draw[<-,dashed,very thick,color=gray!70!white] (5,2.6) to (6,2.6)node[above,xshift=-0.5cm]{$dd$};
		\draw[<-,dashed,very thick,color=gray!70!white] (6,4.7) to (5,4.7)node[above,xshift=0.5cm]{$ad$};

		\node[ellipse,
		draw = black, very thick,text=black,
		minimum width = 3cm, 
		minimum height = 1.2cm,label={[xshift=1.6cm,yshift=-1.65cm,color=black ]$\hat\cX_t^{J,\hat{A},t}$}] (e) at (4,-0.25) {};
		\end{scope}
		\end{tikzpicture}}
	\caption{CPS and the dual process. The CPS $\cX_t^{I,A}$ started with $I=\{2,3\}$ and $A=\{2,4,5,6\}$  running upwards,  and the dual $\hat\cX_s^{t,J,\cA_t}$ started from time $t$ in $J=\{4\}$ and $\mathcal{A}_t=\{2,3,6\}$  running downwards.  The states $(0,a)$, $(0,d)$, $(1,a)$ and $(1,d)$ are depicted with gray solid, gray dotted, black solid and black dashed lines, respectively.}\label{figure:graphical-constr+dual}
\end{figure}

\subsection{Stationary distributions and phase transition}\label{sec:statdis-phasetrans}

In what follows, we will make use of both the map-based as well as the set-based representation for notational convenience.

 Since type switches are independent of the infection and recovery dynamics, it is easy to verify that $\cA_t$ converges in distribution to its unique  stationary distribution $\pi_\alpha$, which lets any site in $S$ be active with probability $\alpha\ldef\frac\rho{\sigma+\rho}$ and dormant otherwise, independently of each other.
Thus, $$\mu^{triv}\ldef\delta_{\{0\}^{\bbZ^d}}\otimes\pi_\alpha$$  is a  \emph{trivial} stationary distribution of $(\xi_t)$. 
 We now investigate the existence, and convergence into, further stationary distributions.

\begin{theorem}\label{Theorem:05}
  Let $(\xi^1_t)$ be a CPS with initial distribution $\delta_{\{1\}^{\bbZ^d}}\otimes\pi_\alpha$.
  Then, $\xi_t^1\Rightarrow \xi_\infty^1$ as $t \to \infty$, where $\cL(\xi_\infty^1)$ is a stationary distribution of $(\xi_t)$  \emph{(the upper invariant measure)}. 
\end{theorem}

  The proof of this theorem will be very similar to that of \cite[(2.7) Theorem, p.\ 123]{D95} and will make use of the following lemma (cf.\ \cite[(2.8) Lemma, p.\ 123]{D95}):
\begin{lemma}\label{Lemma:01b}
  Let $\xi_t^1$ as specified above, then for any sets $H,A,D\subseteq S$ with $A\cap D=\emptyset$ the function
  $$
	\Phi_t(H,A,D)\ldef \mathbb{P}\big(H\subseteq(\cX_t)^c\big|
	    A\subseteq \cA_t, D\subseteq (\cA_t)^c\big)
  $$
  is increasing in $t$.
\end{lemma}
\begin{proof}
  This follows from a restarting argument combined with the attractivity established in Theorem~\ref{thm:mon-add-att}(3):
	Let $H,A,D\subseteq S$ with $A\cap D=\emptyset$ and $s,t>0$. Then,
	\begin{align*}
	  \Phi_{t+s}(H,A,D)
	   &= \bbP\big(H\subseteq(\cX_{t+s})^c\big|A\subseteq\cA_{t+s},D\subseteq(\cA_{t+s})^c\big)\\
	   &= \bbP\big(H\subseteq(\cX_t^{\cX_s})^c\big|A\subseteq\cA_t^{\cA_s},D\subseteq(\cA_t^{\cA_s})^c\big)\\
	   &\geq \bbP\big(H\subseteq(\cX_t^{\bbZ^d})^c
	             \big|A\subseteq\cA_t^{\cA_s},D\subseteq(\cA_t^{\cA_s})^c\big)\\
	   &= \Phi_t(H,A,D),
	\end{align*}
	where we used stationarity of $(\cA_t)$ in the last step.
\end{proof}

\begin{proof}[Proof of Theorem \ref{Theorem:05}]
  We show convergence  of the   finite dimensional distributions  (depending on finitely many sites).  For this, let $m\in\bbN$, $x_1,\dots,x_m\in S$ and $i_1,\dots,i_m\in F$ be fixed. Now, denote by $H=\{x_k\mid 1\leq k\leq m, i_{k,1}=0\}$, the set of the $x_k$ such that $i_k$ is healthy, and $I,A,D$ accordingly. Then,
  \begin{align*}
    \mathbb{P}\big( & \xi_t^1(x_1)=i_1,\dots,\xi_t^1(x_m)=i_m\big)\\
     &= \mathbb P(A\subseteq\cA_t)\cdot\mathbb P(D\subseteq(\cA_t)^c)
         \cdot \mathbb P\big(H\subseteq(\cX^1_t)^c, I\subseteq\cX^1_t
                         | A\subseteq\cA_t, D\subseteq(\cA_t)^c\big),
  \end{align*}
  where the first two factors multiply to $\alpha^{|A|}(1-\alpha)^{|D|}$ and the third can be written  as a finite sum  in terms of $\Phi_t(\bullet,A,D)$ using the inclusion-exclusion-formula. Hence, Lemma~\ref{Lemma:01b} implies the desired convergence.   The stationarity of the  limiting distribution follows from standard arguments.
\end{proof}

 The upper and the lower (trivial) stationary distributions $\cL(\xi^1_\infty)$ and $\mu^{triv}=\delta_{\{0\}^{\bbZ^d}}\otimes\pi_\alpha$ may or may not be distinct. 
This gives rise to the following notion of survival.
\begin{definition}[Survival]\label{def:survival}
  A CPS $(\xi_t)$ is said to
     \emph{survive}, if $\cL(\xi^1_\infty)\neq\mu^{triv}$.
  Otherwise \emph{it goes extinct}.
\end{definition}%
\begin{remark}[Strong survival]\label{rem:strong-surv}
  By a duality argument (cf.\ \cite[p.\ 36]{L99}) combined with the  stationarity of $(\cA_t)$, it can be shown that the above notion of (``infinite'') survival of the CPS is equivalent to ``finite'' survival, i.e.
  $$\bbP(\cX_t^{\{0\},\pi_\alpha}\neq\emptyset\text{ for all }t\geq0)>0.$$
  For the CP it has been shown in \cite[Theorem 1.12]{L99} that this is even equivalent to the generally more restrictive notion of \emph{strong} survival, given as
  $$\bbP(0\in\cX_t^{\{0\},\pi_\alpha}\text{ infinitely often})>0,$$
  which can be very well interpreted as `endemic' behaviour.
  This result, however, is entangled with the proof of \emph{complete convergence}. We conjecture that both also hold for the CPS, but defer the proofs to future work.
\end{remark}

Since the CPS is monotone in the parameters $\lambda_{aa}$, $\lambda_{ad}$, $\lambda_{da}$, $\lambda_{dd}$, $\delta_a$, $\delta_d$, we can  consider  a critical parameter $\lambda_\bullet^c$ or $\delta_\bullet^c$, for some $\bullet$,   separating a survival from an extinction phase in the obvious way,  while  fixing all other parameters. The coupling results in Section \ref{section:coupling} will give us some bounds on the critical parameters and ensure that a phase transition can occur.

There are cases though, where critical parameters do not exist, e.g.\ when $\lambda_{aa}$ is large and $\delta_a$ is small enough. Here,  survival does not depend on recovery rate of the  dormant sites such that there is no critical $\delta_d^c$.

\section{Relation between CPS and CP}\label{sec:comparison}

In this section we pursue two different approaches to compare the CPS with ordinary
CPs.
At first we discuss scaling limits in the switching rates. We show that for high
$\sigma$ and $\rho$ the CPS is well approximated by a CP with  ``effective''  rates $\lambda^*$ and
$\delta^*$ that can be computed explicitly from the model parameters, whereas for low $\sigma$ and $\rho$ the CPS is close to a CP in  a \emph{static environment}.
Secondly, we give rigorous couplings between the CPS and suitable CPs in
Section~\ref{section:coupling},  establishing phase transitions. 
We will discuss various simulation results in Section~\ref{ssn:simulations}.

\subsection{Fast and slow switching limits}\label{sec:switching-limits}
In this section, we will  consider  a rescaled CPS $(\xi_t^h)_{t\ge 0}$ whose  switching rates $\sigma,\rho$ are replaced by suitable $h\sigma$ and $h\rho$, and where we let $h$ either go to $\infty$ or to $0$.

Note that for a CP $(\eta_t)$ the product measure $\cL(\eta_t)\otimes\pi_\alpha$ is for fixed $t$ a distribution on $F^S=\{\{0,1\}\times\{a,d\}\}^{\bbZ^d}$.

\begin{theorem}[Fast switching]\label{thm:fast-switching-limit}
  For $k \in \mathbb{N}$ let $(\xi_t^k)_{t\ge 0}$ be rescaled CPS as above, with switching parameters $\sigma k$, $\rho k$, all other parameters fixed and intial infections given by $I\subseteq S$. Then,  for each fixed $t \ge 0$,
$$
\cL(\xi_t^k) 
\to
\cL(\eta_t)\otimes\pi_\alpha$$
as $k\to\infty$. 
Here, $(\eta_t)$ is a CP with initial infections given by $I$, infection rate
  $$
    \lambda^*
      \ldef \alpha^2\lambda_{aa}
             + \alpha(1-\alpha)(\lambda_{ad}+\lambda_{da})
             + (1-\alpha)^2\lambda_{dd}
  $$
  and recovery rate
  $$
    \delta^*
      \ldef \alpha\delta_a+(1-\alpha)\delta_d.
  $$
  
\end{theorem}

\begin{proof}
As before, we show convergence of the finite-dimensional distributions of $(\xi_t^k)$, that is, of the laws restricted to finite subsets of
  sites $S'\subset S$, $|S'|<\infty$.\\
  Via the graphical construction we couple the random states $\xi_t^k, k \in \mathbb{N}$ in the
  following way: We start by constructing the interaction graph $\Gamma$ of
  $(\xi_s^0)_{0\leq s\leq t}$ originating on the sites $S'$. Note that since the switching rates in this case are 0, there are no type switches involved here. 
  By a standard argument, $\Gamma$
  will almost surely involve only finitely many sites, say $S_\Gamma\supseteq S'$.
  (This mainly involves starting the dual process at time $t$ in $I=S'$ and
  estimating the backwards infection spread by a suitable Yule Process.)
  Then, we iteratively construct $(\xi_t^{k+1})$ from $(\xi_t^k)$ for $k\geq0$
  by adding a new pair of independent Poisson process of rates $\sigma$ and $\rho$ to each line
  $s\in S_\Gamma$ providing new $\sigma$- and $\rho$-events.\\
  Since there are a.s.\ only finitely many arrows and $\delta$-events in $\Gamma$, say at times $t_1,\ldots,t_{n_\Gamma}$, the activity states at those times and at any
  $s\in S_\Gamma$, $(\xi_{t_i}^k(s)_2)_{i=1,\ldots,n_\Gamma}$ converge in
  distribution as $k \to \infty$ to $n_\Gamma$ independent Bernoulli-$\alpha$ random variables.\\
  Hence, outside an event with vanishing probability conditioned on $\Gamma$,
  the infection process flowing through the graphical construction of
  $\xi_t^k$ on $S_\Gamma$ behaves as an infection flowing through $\Gamma$,
  stopping at $\delta_a$- and $\delta_d$-events with probability $\alpha$ and
  $1-\alpha$ respectively and passing through arrows of type $aa$ and $dd$ with 
  probability $\alpha^2$ and $(1-\alpha)^2$ and through arrows of types $ad$ and
  $da$ with probability $\alpha(1-\alpha)$. This now is equal in distribution to
  the desired contact process $(\eta_t)$ with rates $\lambda^*$ and $\delta^*$.\\
  
  Finally, we obtain the desired convergence by taking expectations of the conditional
  probabilites given $\Gamma$ and considering dominated convergence.
\end{proof}

Simulations suggest that the approximation in this theorem already works  rather well for relatively  low values of $k$ (cf.\ Figure~\ref{figure:CPS-bounds}).

\bigskip

Similarly we can show that  for a suitable  limit $h\to0$, the family $(\xi_t^h)$ converges to a CP in \emph{static environment}, by which we mean that activity states are only
random in the initial state $\cA_0$ and then remain constant over time.

\begin{theorem}[Slow switching]\label{thm:slow-switching-limit}
  Let $(\xi_t^{1/k})$ as above, with switching parameters $\sigma/k$, $\rho/k$, all other parameters fixed, and initial infections given by $I\subseteq S$. Then, for each fixed $t \ge 0$,
$$
\cL(\xi_t^k) 
\to
\cL(\xi^0_t)$$
as $k\to\infty$.   
Here, $(\xi^0_t)$ is a CPS with switching rates $\rho=\sigma=0$, hence it can be considered as a CP in a static environment determined by $\mathcal{A}_0$ and initially infected individuals given by $I$.
\end{theorem}
\begin{proof}
  This can be proved analogously to the previous theorem by first building the
  graphical construction of $\xi_t^1$ and then successively thinning out the
  $\sigma$- and $\rho$-events, keeping those events only with probabilities
  $\frac k{k+1}$, to obtain $\xi_t^{1/(k+1)}$ from $\xi_t^{1/k}$.
\end{proof}
\label{section:CPSRE}

 Note that the CP in static environment has been studied comprehensively over the last decades, cf.\  eg.\  Bramson, Durrett and Schonmann \cite{BDS91} or Klein \cite{Kl94}.  In these references the assumption is made that the intensities $\lambda(x,y)$ for the Poisson processes indicating the infections from $x$ to $y$ are independent and identically distributed for all  neighboring  pairs of $x,y\in S$. 
However, considering the CPS-limit obtained above,  this is possible only in the special case where the infection rates do not depend on the activity state at all, i.e.\ when the original CPS is a CPREE. 

To the author's best  knowledge, the scaling limits of Theorems~\ref{thm:fast-switching-limit} and \ref{thm:slow-switching-limit} have not yet been shown for the CPREE. However, Broman conjectured that $\xi_t^k$ behaves like the ordinary CP $(\eta_t)$ for large $k$ in \cite[Remark after Proposition 1.9]{Broman2007}.

\begin{remark}
  Although the iid condition of \cite{BDS91} and \cite{Kl94} is not satisfied for the static version of the CPD, one can still give conditions for survival and extinction:
  First of all note that the duality of Theorem~\ref{thm:duality} still holds and hence, survival is equivalent to finite survival (cf.\ Remark~\ref{rem:strong-surv}).
  Finite survival is heavily linked to the existence of an infinite percolation cluster containing $0_{\bbZ^d}$ in the sense that, if $\cA_0$ is distributed according to $\pi_\alpha$ and $\alpha$ is smaller than the critical value $p_c$ for site percolation on $\bbZ^d$, the infection goes extinct almost surely.
  However, in the extreme case of $\delta_d=0$, whenever $\alpha<1$ there is finite survival since in the event $0_{\bbZ^d}\not\in\cA_0$ the particle at $0_{\bbZ^d}$ will never recover.
\end{remark}

\subsection{Couplings between CPS and the classical Contact Process}\label{section:coupling}

In this section we employ and extend the methods and results of \cite{Broman2007} to construct couplings between CPS and CP.
For this, we will make heavy use of and build upon \cite[Theorem 1.4]{Broman2007}, which shows that an inhomogeneous Poisson Point Process $P$ with rate randomly switching between two fixed values, say $\mathfrak a_0\leq\mathfrak a_1$, can be coupled to a homogeneous Poisson Point Process $L$ with some (explicit) rate $\bar{\mathfrak a}\in(\mathfrak a_0,\mathfrak a_1)$ such that $L$ is dominated by $P$, in the sense that $L_t\subseteq P_t$ for all $t$, and where $\bar{\mathfrak a}$ is the largest value with this property. Unfortunately, a corresponding dominating Poisson Point Process $U$, with $U_t\supseteq P_t$ for all $t$, can only exist in trivial cases, i.e.\ for rates of at least $\mathfrak a_1$.

When we use this coupling argument for the point process of $\delta$-events that are not blocked by activity states -- occurring with rate $\delta_a$ while an individual is active and rate $\delta_d$ while an individual is dormant, and further use the trivial upper bound for the infection rates, we obtain a dominating CP for our CPS as follows:

\begin{theorem}[Dominating CP]\label{thm:upper-bound-CP}
  Let $(\xi_t)$ be a CPS where the initial distribution is of the form $\mu=\nu\otimes\pi_\alpha$ and define
  \begin{align*}
    \lambda_{\max}
      &\ldef \max{\{\lambda_{aa},\lambda_{ad},\lambda_{da},\lambda_{dd}\}},\\[.5em]
    \bar\delta
      &\ldef \frac12\big(\delta_a+\delta_d+\rho+\sigma
              - \sqrt{(|\delta_a-\delta_d|-\rho-\sigma)^2
                + 4\sigma|\delta_a-\delta_d|}\big)
       \geq0.
  \end{align*}
  Then $(\xi_t)$ gets dominated by a basic CP $(\eta_t)$ with infection rate $\lambda_{\max}$ and recovery rate $\bar\delta$, i.e.\ we can couple the processes via the graphical construction such that for any initial configuration of infected particles $I\subseteq S$ we have
  \begin{gather*}
    \cX_t^I\subseteq\cZ_t^{I}\quad\text{for all}\quad t\geq 0,
  \end{gather*}
  where $\cZ_t^I$ denotes the set of infected particles of $\eta_t^I$.
\end{theorem}

\begin{proof}
  By monotonicity (cf.\ Theorem~\ref{thm:mon-add-att}) it is clear that $(\xi_t)$ is dominated by another CPS $(\hat\xi_t)$ with the same $\delta$-rates but $\hat\lambda_\bullet=\lambda_{\max}$ for all $\bullet\in\{aa,ad,da,dd\}$, which is in fact a CPREE. Now, the above theorem follows from \cite[Theorem 1.6]{Broman2007}, i.e.\ by applying \cite[Theorem 1.4]{Broman2007} for any $s\in S$ on the $\delta$-events on $s$.
  The non-negativity of $\bar\delta$ holds, since
  \begin{align*}
    (\delta_a+\delta_d+\rho+\sigma)^2
     &= (2\min\{\delta_a,\delta_d\}+|\delta_a-\delta_d|+\rho+\sigma)^2\\
     &\geq (|\delta_a-\delta_d|+\rho+\sigma)^2\\
     &= (|\delta_a-\delta_d|-\rho-\sigma)^2 + 4(\rho+\sigma)|\delta_a-\delta_d|.
  \end{align*}
\end{proof}

\begin{remark}\label{remark:distancing-dormancy}
  As mentioned in the introduction, there are two natural intepretations for the CPD:  (microbial) host dormancy with $\delta_d=0$, and  social distancing interpreted as (human) dormancy with $\delta_d=\delta_a$. Since for $\rho\to0$ it holds $\bar\delta\to\min\{\delta_a,\delta_d\}$, one can see by Theorem~\ref{thm:upper-bound-CP} leads to significantly different bounds. Also, the fast switching approximation from Theorem~\ref{thm:fast-switching-limit} can differ strongly for these two strategies. Indeed, in both cases the human dormancy is favorable.
\end{remark}

For a lower bound we now strive to apply Broman's coupling to the infection arrows originating from a single particle, instead of to the recovery events. To account for the stochastic dependencies between such arrows, we need to extend \cite[Theorem 1.4]{Broman2007}:

\begin{lemma}\label{lem:broman-extension}
  For $k\in\bbN$ let $N^{1,0},\dots,N^{k,0}$ be independent Poisson Point Processes (PPPs) with rate $\lambda_0$ and $N^{1,1},\dots,N^{k,1}$ independent PPPs with rate $\lambda_1$ such that $N^{i,0}$ is also independent of $N^{j,1}$ for all $i,j\in[k]$ with $i\neq j$. Further, let $(J_t)$ 
  be a jump process in $\{0,1\}$ jumping to $0$ at rate $\sigma$ and to $1$ at rate $\rho$. Then, there are $k$ independent PPPs $L^{1},\dots,L^{k}$ with rate
  \begin{gather*}
    \bar\lambda(\lambda_0,\lambda_1,\sigma,\rho,k)
      \ldef \frac12\bigg(
              \lambda_1 + \lambda_0 + \tfrac\rho k+\tfrac\sigma k
               - \sqrt{(|\lambda_1-\lambda_0|-\tfrac\rho k - \tfrac\sigma k)^2
                  + \tfrac{4\sigma}k|\lambda_1-\lambda_0|}\bigg)
  \end{gather*} 
  such that almost surely for any $i\in[k]$ it holds
  $$
    L^i
      \subseteq\{\tau\in N^{i,j}\mid J_\tau=j\}.
  $$
\end{lemma}

\begin{proof}
  W.l.o.g. assume that $\lambda_1\geq\lambda_0$.
  We apply \cite[Theorem 1.4]{Broman2007} to the PPPs $$N^i:=\bigcup_{j=1}^kN^{j,i}$$ which are of rate $k\cdot\lambda_i$ respectively, delivering a PPP $L$ of rate $k\cdot\bar\lambda$ such that any $\tau\in L$ holds $\tau\in N^{J_\tau}$. (Note that for $k=1$ this already concludes the proof.)

  Now, noting that the random variables $U_\tau$ for $\tau\in N^0\cup N^1$, given by $U_\tau=i\in[k]$ iff $\tau\in N^{i,j}$ for some $j$, are iid, uniformly distributed on $[k]$ and independent of $(J_t)$, it is straightforward to subdivide $L$ into $L^1,\ldots,L^k$ such that the Lemma holds.
\end{proof}

To illustrate the lemma above and the strategy of the following proof, we provide Figure~\ref{figure:approxarrows}.
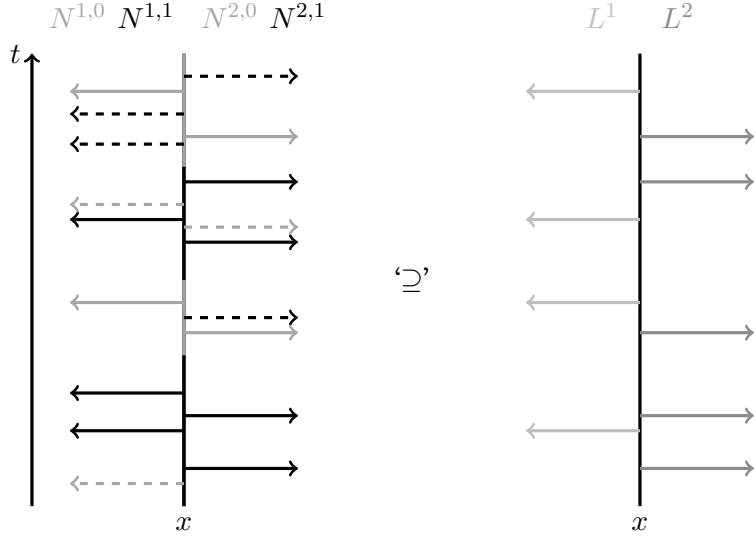
\begin{figure}[h]
		\centering
		\begin{tikzpicture}[darkstyle/.style={circle,fill=black,inner sep=1pt}]
		
		\draw[->, very thick] (-2,0)--(-2,6)node[left] {$t$};
		\draw[-, very thick] (0,6)--(0,0)node[below] {$x$};
		\node () at (-1.4,6.5){\textcolor{gray!70!white}{$N^{1,0}$}};
		\node () at (-0.5,6.5){\textcolor{black}{$N^{1,1}$}};
		\node () at (1.5,6.5){\textcolor{black}{$N^{2,1}$}};
		\node () at (0.6,6.5){\textcolor{gray!70!white}{$N^{2,0}$}};
		\draw[-, very thick,color=black] (0,0)--(0,2)node[left] {};
		\draw[-, very thick,color=gray!70!white] (0,2)--(0,3)node[left] {};
		\draw[-, very thick,color=black] (0,3)--(0,4.5)node[left] {};
		\draw[-, very thick,color=gray!70!white] (0,4.5)--(0,6)node[left] {};
		\draw[->,very thick,color=black] (0,0.5) to (1.5,0.5)node[above,xshift=0.5cm]{};
		\draw[->,very thick,color=black] (0,1.2) to (1.5,1.2)node[above,xshift=0.5cm]{};
		\draw[->,very thick,dashed,color=black] (0,2.5) to (1.5,2.5)node[above,xshift=0.5cm]{};
		\draw[->,very thick,color=black] (0,3.5) to (1.5,3.5)node[above,xshift=0.5cm]{};
		\draw[->,very thick,color=black] (0,4.3) to (1.5,4.3)node[above,xshift=0.5cm]{};
		\draw[->,very thick,dashed,color=black] (0,5.7) to (1.5,5.7)node[above,xshift=0.5cm]{};

		\draw[->,very thick,color=black] (0,1) to (-1.5,1)node[above,xshift=0.5cm]{};
		\draw[->,very thick,color=black] (0,1.5) to (-1.5,1.5)node[above,xshift=0.5cm]{};
		\draw[->,very thick,color=black] (0,3.8) to (-1.5,3.8)node[above,xshift=0.5cm]{};
		\draw[->,very thick,dashed,color=black] (0,4.8) to (-1.5,4.8)node[above,xshift=0.5cm]{};
		\draw[->,very thick,dashed,color=black] (0,5.2) to (-1.5,5.2)node[above,xshift=0.5cm]{};
		
		\draw[->,very thick,dashed,color=gray!70!white] (0,0.3) to (-1.5,0.3)node[above,xshift=0.5cm]{};
		\draw[->,very thick,color=gray!70!white] (0,2.7) to (-1.5,2.7)node[above,xshift=0.5cm]{};
		\draw[->,very thick,dashed,color=gray!70!white] (0,4) to (-1.5,4)node[above,xshift=0.5cm]{};
		\draw[->,very thick,color=gray!70!white] (0,5.5) to (-1.5,5.5)node[above,xshift=0.5cm]{};
		
		\draw[->,very thick,color=gray!70!white] (0,2.3) to (1.5,2.3)node[above,xshift=0.5cm]{};
		\draw[->,very thick,dashed,color=gray!70!white] (0,3.7) to (1.5,3.7)node[above,xshift=0.5cm]{};
		\draw[->,very thick,color=gray!70!white] (0,4.9) to (1.5,4.9)node[above,xshift=0.5cm]{};

		\node () at (3,3){`$\supseteq$'};
		\begin{scope}[shift={(6,0)}]
			\draw[-, very thick] (0,6)--(0,0)node[below] {$x$};
			\node () at (-0.5,6.5){\textcolor{gray!50!white}{$L^1$}};
			\node () at (0.5,6.5){\textcolor{gray!90!white}{$L^2$}};
			\draw[->,very thick,color=gray!90!white] (0,0.5) to (1.5,0.5)node[above,xshift=0.5cm]{};
			\draw[->,very thick,color=gray!90!white] (0,1.2) to (1.5,1.2)node[above,xshift=0.5cm]{};
			\draw[->,very thick,color=gray!90!white] (0,4.3) to (1.5,4.3)node[above,xshift=0.5cm]{};
			\draw[->,very thick,color=gray!90!white] (0,2.3) to (1.5,2.3)node[above,xshift=0.5cm]{};
			\draw[->,very thick,color=gray!90!white] (0,4.9) to (1.5,4.9)node[above,xshift=0.5cm]{};

			\draw[->,very thick,color=gray!50!white] (0,1) to (-1.5,1)node[above,xshift=0.5cm]{};
			\draw[->,very thick,color=gray!50!white] (0,3.8) to (-1.5,3.8)node[above,xshift=0.5cm]{};
			\draw[->,very thick,color=gray!50!white] (0,2.7) to (-1.5,2.7)node[above,xshift=0.5cm]{};
			\draw[->,very thick,color=gray!50!white] (0,5.5) to (-1.5,5.5)node[above,xshift=0.5cm]{};
		\end{scope}
		\end{tikzpicture}
		\caption{This figure illustrates Lemma \ref{lem:broman-extension} for $k=2$ in the context of the graphical construction. We interpret the states $0$ and $1$ of $J_t$ as the activity of a site $x\in S$ at time $t$, indicated by the colors gray and black, respectively. Gray arrows belong to the PPPs $N^{i,0}$, black ones to $N^{i,1}$. They are depicted solid iff they fit the activity of $x$. Arrows to the left belong to $N^{1,j}$, those to the right to $N^{2,j}$. The lemma aims to estimate the solid arrows (black and gray) pointing to the left from below (in the sense of\ `$\subseteq$') by a PPP $L^1$ and the ones pointing to the right by $L^2$ (as depicted on the right) such that $L^1$ and $L^2$ are independent.}\label{figure:approxarrows}
\end{figure}
Note that in the above lemma, the PPPs $N^{i,0}$ and $N^{i,1}$ for fixed $i$ do not need to be independent, which leaves room for the above mentioned dependencies. With this, we can now provide a complementary result to Theorem \ref{thm:upper-bound-CP}, involving a dominated CP, by taking the trivial bound for the $\delta$-rates and applying Lemma~\ref{lem:broman-extension} to the $\lambda$-rates.

\begin{theorem}[Dominated CP]\label{thm:lower-bound-CP}
  Let $(\xi_t)$ be a CPS whose initial distribution is of the form $\mu=\nu\otimes\pi_\alpha$. Further, we define
  \begin{align*}
    \bar\lambda
     &\ldef 
       \frac12\bigg(\lambda_a+\lambda_d+\tfrac{\rho+\sigma}{|\cN|}
          - \sqrt{\big(|\lambda_a-\lambda_d|-\tfrac{\rho+\sigma}{|\cN|}\big)^2
              + \tfrac{4\sigma}{|\cN|}|\lambda_a-\lambda_d|}\bigg)
      \geq0
  \end{align*}
  with $\lambda_a=\min\{\lambda_{aa},\lambda_{ad}\}$
  and $\lambda_{d}=\min\{\lambda_{da},\lambda_{dd}\}$.
  Then $(\xi_t)$ dominates a basic CP $(\eta_t)$ with infection rate $\bar\lambda$ and recovery rate $\delta_{\max}=\max\{\delta_a,\delta_d\}$.
\end{theorem}
\begin{proof}
  First of all, w.l.o.g.\ assume that $\delta_a=\delta_d=\delta_{\max}$, $\lambda_{aa}=\lambda_{ad}=\lambda_a$ and $\lambda_{da}=\lambda_{dd}=\lambda_d$. Otherwise, consider a CPS $(\bar\xi_t)$ with these rates, which is dominated by $(\xi_t)$ by monotonicity. Further, assume that $\lambda_a\geq\lambda_d$ or exchange the roles of $a$ and $d$.
  
  Let $x\in S$ and enumerate $\{y\in S\mid y\sim x\}=\{y_1,\ldots,y_{|\cN|}\}$. Further, for $i\in[|\cN|]$ let $A_i=\{t\geq0\mid y_i\in \cA_t\}$, the active times of $y_i$, and define
  \begin{align*}
    N^{i,0}
     &\ldef (T^{x\to y,da}\cap A_i) \cup (T^{x\to y,dd}\cap A_i^c)
    \text{ as well as }\\
    N^{i,1}
     &\ldef (T^{x\to y,aa}\cap A_i) \cup (T^{x\to y,ad}\cap A_i^c).
  \end{align*}
  Now, both $(N^{i,0})$ and $(N^{i,1})$ are collections of PPPs with rates $\lambda_d$ and $\lambda_a$ respectively, such that, letting $J_t=\indicator_{\cA_t}(x)$, Lemma~\ref{lem:broman-extension}, provides 
  independent PPPs $(L^{x\to y})_{y\sim x}$ each with rate $\bar\lambda$, dominated by the infection arrows of the CPS in the desired sense. Repeating this for every $x\in S$, using all the resulting $L^{x\to y}$ to build a graphical construction, delivers the desired dominated CP with infection rate
$\bar\lambda(\lambda_d,\lambda_a,\sigma,\rho,|\cN|)$. Insertion concludes the proof.
(Non-negativity of $\bar\lambda$ follows analogously to that of $\bar\delta$.)
\end{proof}

\begin{remark}[Sharpness and further improvements of the bounds]\label{remark:sharpness}
Trivially, both the dominating and dominated CP bounds become sharp when $\lambda_\bullet$ and $\delta_\bullet$ are constant, i.e.\ when the CPS is a CP itself.

  Further, for extreme values of $\sigma$ and $\rho$ the bounds can also be sharp, since for $\rho>|\delta_a-\delta_d|$ as $\sigma\to 0$ it holds that $\bar\delta\to\delta_{\max}$. 
  This implies that the coupling is good when dormant states are lost and the CPREE becomes a CP itself. In particular, this shows that $\bar\delta$ is a much better bound than the trivial bound $\min\{\delta_a,\delta_d\}$. Similar arguments can be made for $\bar\lambda$.

  Adapting the proof above slightly, considering \emph{incoming} arrows instead of outgoing ones, one can also choose $\lambda_a=\min\{\lambda_{aa},\lambda_{da}\}$ and $\lambda_{d}=\min\{\lambda_{ad},\lambda_{dd}\}$. This may be used to improve the value of $\bar\lambda$ and thus to further refine the bounds obtained from Theorem~\ref{thm:lower-bound-CP}.
\end{remark}

\begin{remark}[Couplings for the CPD]\label{remark:CPD}
  In special cases such as the CPD, these bounds can be very bad: Here, the dominated CP, which has infection rate $\bar\lambda=0$ and recovery rate $\delta_a>0$, will never survive when started with finitely many infections.
  However, in order to ensure survival one can construct an alternative coupling between the CPD (with $\delta_d=0$ or $\delta_d=\delta_a$) and oriented percolation by adapting the arguments in \cite[p.\ 138-142]{D95} to our setting. This gives the existence of a critical $\delta^c$ for fixed $\sigma$, $\rho$ and $\lambda$.
\end{remark}

\begin{remark}[Couplings for the CPREE]
\label{remark:CPREE}
  For the CPREE the dominating CP with $\lambda_{\max}=\lambda$ and $\bar\delta$ from Theorem~\ref{thm:upper-bound-CP} looks promising. Indeed, if we let $\sigma,\rho\to\infty$ simultaneously, we observe that
  \[
    \bar\delta \to \alpha\delta_a+(1-\alpha)\delta_d,
  \]
  which coincides with $\delta^*$ from Theorem~\ref{thm:fast-switching-limit}.
  See Figure~\ref{figure:CPREE} for a comparison via simulation.
\end{remark}

\begin{remark}[Non-existence of critical infection rate for the CPD]
\label{remark:CPD2}
Note that in the special case of the CPD, there are choices of $\sigma$, $\rho$ and $\delta_d$ such that the process does not survive for any choice of $\lambda$ and $\delta_a$, i.e.\ $\lambda^c(\sigma,\rho,\delta_d,\delta_a)=\infty$ for all $\delta_a$. This can be shown similarly to \cite[p.\ 2401-2403]{R08} by making a connection to continuum percolation and giving a coupling with a multi-type-branching process.
\end{remark}

\begin{remark}[Monotonicity in the switching parameters]\label{remark:mon-switching} In some special cases one can even show monotonicity in the parameters $\sigma$ and $\rho$ via the graphical construction.
For example, it is straightforward to couple two CPREE's $(\cX_t,\cA_t)(\sigma_1,\rho_1,\delta_a,\delta_d,\lambda)$ and $(\cY_t,\cB_t)(\sigma_2,\rho_2,\delta_a,\delta_d,\lambda)$ with $\sigma_1\leq\sigma_2$, $\rho_1\geq\rho_2$ and $\delta_d\leq\delta_a$ such that we have
$$
\cB_t\subseteq\cA_t \quad\text{ and }\quad \cX_t\subseteq \cY_t\quad\text{ for all }t\geq 0.
$$
This comes from the fact that the dormant state is only in favor for the infection because it decreases the recovery rate of the particles (if $\delta_d\leq\delta_a$) and does not harm the spread of the infection.

On the other hand, for the CPD with $\delta_a=\delta_d$ the dormant state is harmful for the infection in the sense that the recovery rate does not decrease and some infections are not allowed. Hence, we can show monotonicity in $\sigma$ and $\rho$ for this special case as well.
\end{remark}

\subsection{Simulations}\label{ssn:simulations}

Figure~\ref{figure:CPS-bounds} illustrates all results of Sections~\ref{sec:switching-limits} and \ref{section:coupling} in comparison to a single CPS (top left). This shows that, even when $\sigma$ and $\rho$ are well within the range of the other parameters, the fast switching limit CP (top right) can give an apparently reasonably good approximation to the CPS, while neither dominating CP (below left) nor dominated CP (below right) nearly come as close.

Meanwhile, Figure~\ref{figure:CPREE} shows, as argued in Remark~\ref{remark:CPREE}, that the dominating CP (right) can be a good approximation for the CPREE (left). Note that here with $\sigma=\rho=5$, neither is switching particularly fast, nor extreme in the sense of Remark~\ref{remark:sharpness}, where $\rho\gg\sigma$.

Finally, Figure~\ref{figure:strongSurv} suggests that the critical parameter $\lambda^c$ for (finite) survival of the CPD under consideration lies in the interval $[7,8]$.
Since in that particular process sites are active half of the time on average, intuitively one would assume that roughly a quarter of infection arrows are successful. This argument suggests that $\lambda^c\approx4\lambda^c_{CP}$, where $\lambda^c_{CP}$ denotes the critical parameter of a CP with recovery rate $\frac12$. (Note that this corresponds to the effective rates from Theorem~\ref{thm:fast-switching-limit}.) However, as shown in \cite[Theorem 1.3]{L95}, $4\lambda^c_{CP}\leq 4\cdot\frac12\cdot 1.95 = 3.9$. Thus, intuition is off here roughly by a factor of $2$.
It is reasonable to assume that this factor increases significantly as switching gets slower and the CPS gets closer and closer to a CP in static environment.

\begin{figure}
	\hspace*{-1.25cm}%
	\begin{subfigure}{.65\textwidth}
		\includegraphics[width=0.88\linewidth,trim= {0 0 3.01cm 0},clip]{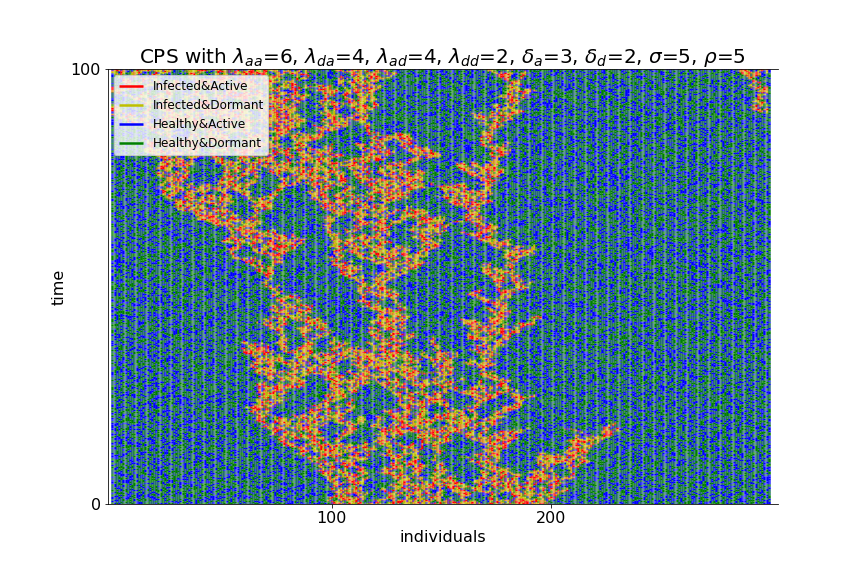}%
	\end{subfigure}\hspace*{-2.5em}
	\begin{subfigure}{.65\textwidth}
		\includegraphics[width=0.88\linewidth,trim= {3.01cm 0 0 0},clip ]{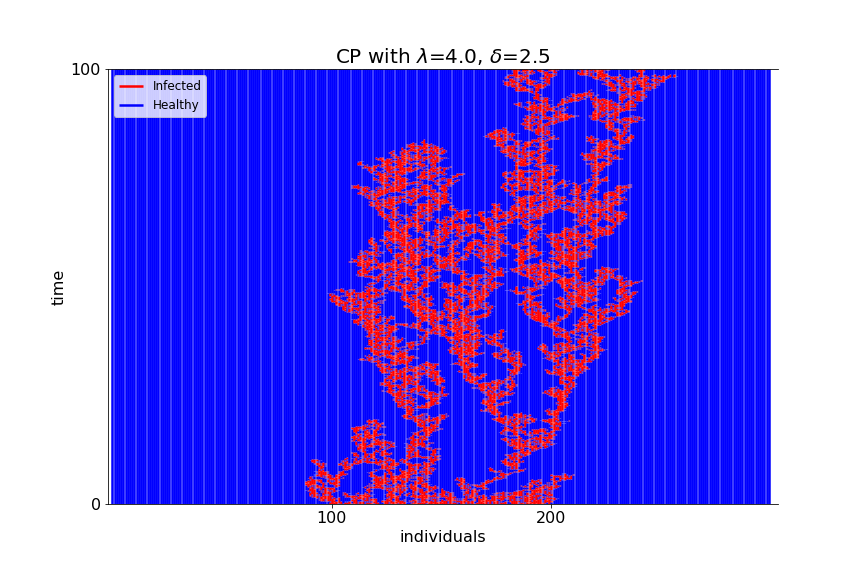}%
	\end{subfigure}%
    \phantom.\\
	\hspace*{-1.25cm}%
	\begin{subfigure}{.65\textwidth}
		\includegraphics[width=0.88\linewidth,trim= {0 0 3.01cm 0},clip]{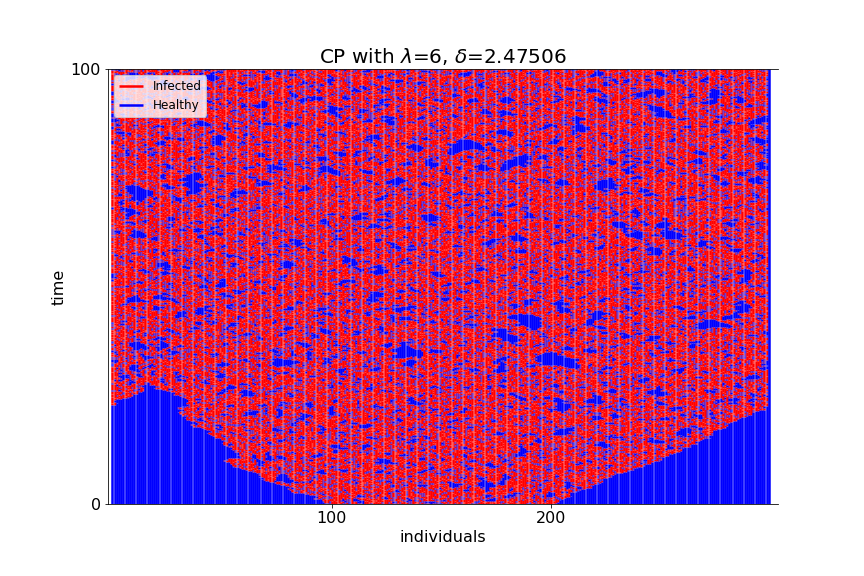}%
	\end{subfigure}\hspace*{-2.5em}
	\begin{subfigure}{.65\textwidth}
		\includegraphics[width=0.88\linewidth,trim= {3.01cm 0 0 0},clip ]{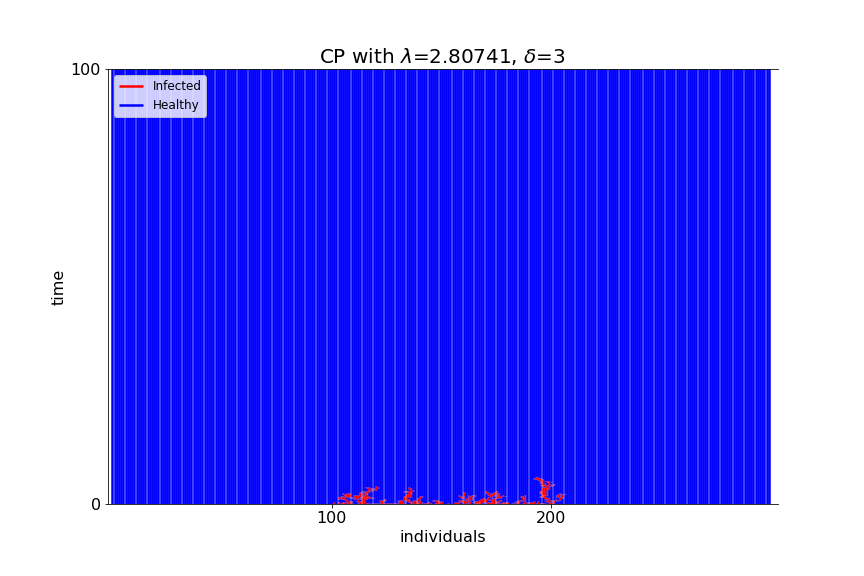}%
	\end{subfigure}%
	\caption{
	  CPS (top left), fast switching limit CP ($\lambda^*,\delta^*)$
	  (top right), dominating CP ($\lambda_{\max},\bar\delta$) (below left) and
	  dominated CP ($\bar\lambda,\delta_{\max}$) (below right). The colors are
	  chosen in the same way as in Figure~\ref{figure:graphical-constr+dual}.
	}
	\label{figure:CPS-bounds}
\end{figure}

\begin{figure}
	\hspace*{-1.cm}%
	\begin{subfigure}{.65\textwidth}
		\includegraphics[width=0.88\linewidth,trim= {0 0 3.01cm 0},clip]{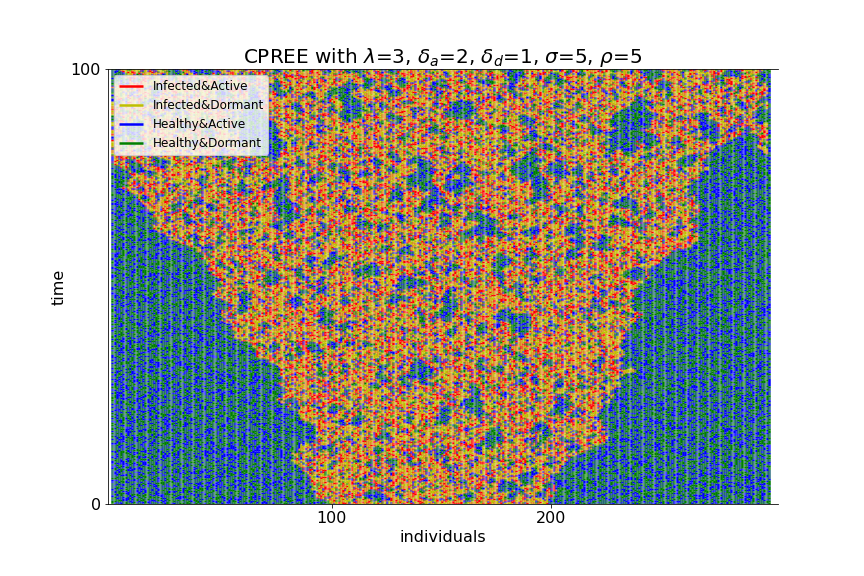}%
	\end{subfigure}\hspace*{-2.5em}
	\begin{subfigure}{.65\textwidth}
		\includegraphics[width=0.88\linewidth,trim= {3.01cm 0 0 0},clip ]{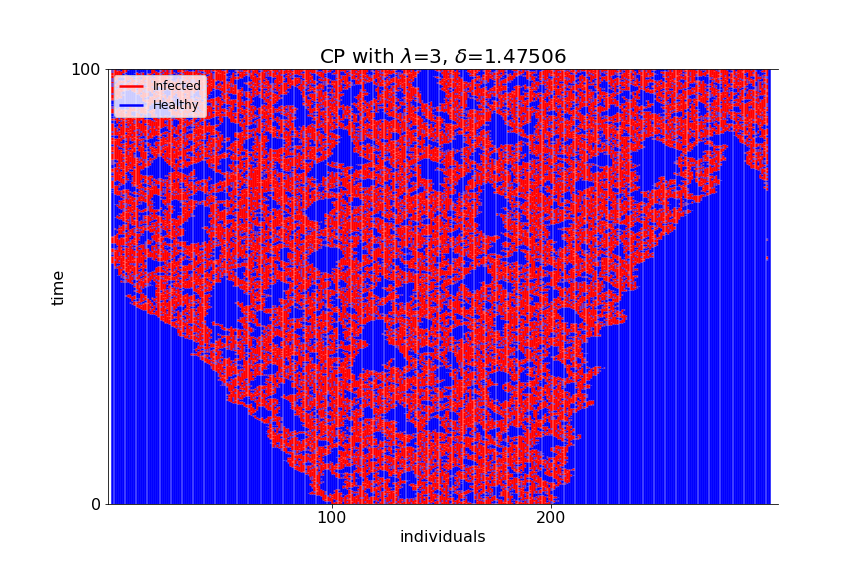}%
	\end{subfigure}%
	\caption{
	  CPREE (left), dominating CP ($\lambda_{\max},\bar\delta$) (right).
	}
	\label{figure:CPREE}%
\end{figure}

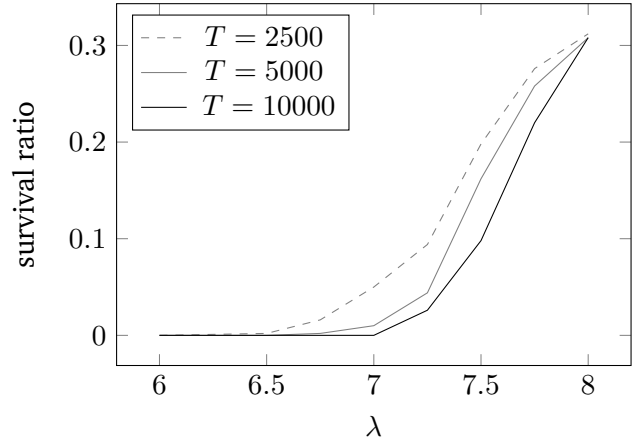
\begin{figure}
  \begin{tikzpicture}
    \begin{axis}
      [domain=6:8,
			width=0.6\textwidth,
			height=0.3\textheight,
			xlabel=$\lambda$,
			ylabel=survival ratio,
            legend pos=north west
	  ]
      \addplot[gray,dashed] table [x=L, y=short, col sep=comma] {strong_surv_data.csv};
      \addlegendentry{$T=2500$}
      \addplot[gray] table [x=L, y=moderate, col sep=comma] {strong_surv_data.csv};
      \addlegendentry{$T=5000$}
      \addplot[black] table [x=L, y=long, col sep=comma] {strong_surv_data.csv};
      \addlegendentry{$\,\,T=10000$}
    \end{axis}
  \end{tikzpicture}
  \caption{
    Simulations of a CPD, $\sigma=\rho=1$, $\delta_a=\delta_d=\frac12$, on $\bbZ/400\bbZ$ with only one initially infected individual at $\{0\}$. The plot shows for $\lambda\in\{6,6.5,6.75,7,7.25,7.5,7.75,8\}$ how many infections of 500 iterations survived until time $T$.
  }
  \label{figure:strongSurv}
\end{figure}

\section{Relation to other models and open problems}

\subsection{Relation to the `classical' multi-type Contact Process}

The classical multi-type (in our set-up: two-type) Contact Process (MTCP) as described e.g.\ in \cite{L85, N92} can be obtained as a scaling limit of our model. Recall that the MTCP does not allow for switches between types of infected individuals, and the healthy state does not carry a type at all. This means that if one only allows for finite switching rates $\rho, \sigma$, the classical multi-type Contact Process lies outside of our modelling framework. However, if one is willing to consider the limits of infinite switching rates (in a suitable sense), the model can be accommodated via the transitions in Figure~\ref{figure:MTCP-CPB} (left).

\subsection{Relation to the CPB of Remenik}
Remenik in \cite{R08} introduced a model which is closely related to the CPD and which we will call the Contact Process with Blocking (CPB), see Figure~\ref{figure:MTCP-CPB} (right). It can be seen as a special case of our model with the convention $\delta_a=1$ and $\delta_d=\infty$. (Choose $\sigma=\delta$ and $\rho=\alpha\delta$ for some $\alpha,\delta>0$ to obtain the nomenclature of \cite{R08}.)
This model is interesting because we can couple any CPS $(\cX_t,\cA_t)$ (with recovery rate $\delta_a$ scaled to one) with a CPB $(\cY_t,\cB_t)$ with identical switching rates and infection rate $\lambda=\lambda_{aa}$ via the graphical construction such that the CPB gets dominated by the CPS, i.e.\ $\cX_t\supseteq\cY_t$. Further, Remenik was able to show complete convergence for his model (cf.\ \cite[Proposition~5.1]{R08}).

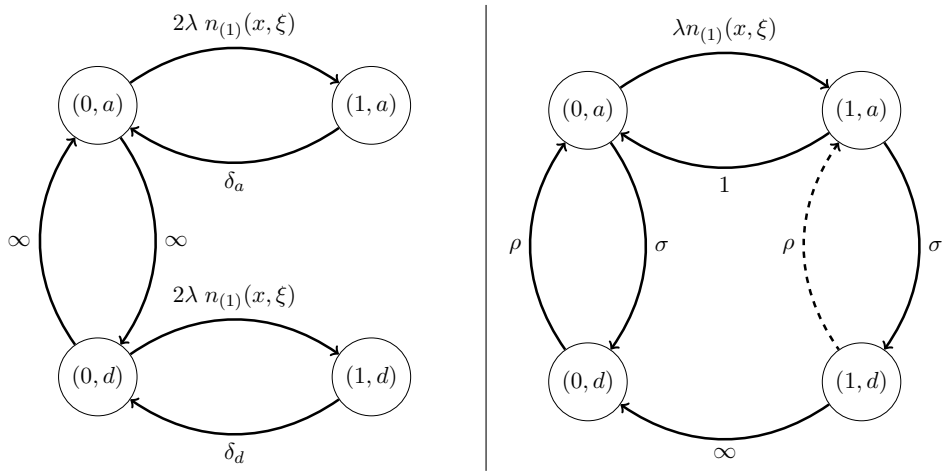
\begin{figure}[h]
	\centering
	\begin{minipage}{0.45\textwidth}
	  \scalebox{0.8}{
		\begin{tikzpicture}
		\node (A) at (0,0) [circle,draw] {$(0,a)$};
		\node (B) at (4.5,0) [circle,draw]{$(1,a)$};
		\node (C) at (0,-4.5) [circle,draw]{$(0,d)$};
		\node (D) at (4.5,-4.5) [circle,draw]{$(1,d)$};
		
		\draw[->, very thick] (A) to[bend left=35]node[above] {$2\lambda~n_{(1)}(x,\xi)$} (B);
		\draw[->, very thick] (C) to[bend left=35]node[above] {$2\lambda~n_{(1)}(x,\xi)$} (D);
		\draw[<-, very thick] (A) to[bend right=35]node[below] {$\delta_a$} (B);
		\draw[<-, very thick] (C) to[bend right=35]node[below] {$\delta_d$} (D);
		\draw[->, very thick] (A) to[bend left=35]node[right] {$\infty$} (C);
		\draw[<-, very thick] (A) to[bend right=35]node[left] {$\infty$} (C);
		\end{tikzpicture}}
	\end{minipage}
	\ \vrule\ 
	\begin{minipage}{.45\textwidth}
	  \scalebox{.8}{
		\begin{tikzpicture}
		\node (A) at (0,0) [circle,draw] {$(0,a)$};
		\node (B) at (4.5,0) [circle,draw]{$(1,a)$};
		\node (C) at (0,-4.5) [circle,draw]{$(0,d)$};
		\node (D) at (4.5,-4.5) [circle,draw]{$(1,d)$};
		
		\draw[->, very thick] (A) to[bend left=35]node[above] {$\lambda n_{(1)}(x,\xi)$} (B);
		\draw[<-, very thick] (A) to[bend right=35]node[below] {$1$} (B);
		\draw[<-, very thick] (C) to[bend right=35]node[below] {$\infty$} (D);
		\draw[->, very thick] (A) to[bend left=35]node[right] {$\sigma$} (C);
		\draw[<-, very thick] (A) to[bend right=35]node[left] {$\rho$} (C);
		\draw[->, very thick] (B) to[bend left=35]node[right] {$\sigma$} (D);
		\draw[<-, very thick, dashed] (B) to[bend right=35]node[left] {$\rho$} (D);
		
		\end{tikzpicture}}
	\end{minipage}
	\caption{Left: The multi-type Contact Process; switching rates $\sigma_0=\rho_0 \to \infty$, $\sigma_1=\rho_1=0$ (MTCP).  Right: State transition graph of the Contact Process with blocking (CPB); $\delta_d\to\infty$, dashed transition will never occur.}
	\label{figure:MTCP-CPB}
\end{figure}

\subsection{The CP with infection dormancy (CPID)}\label{sec:CPID}

Besides the host dormancy perspective, there is also an infection dormancy perspective (``latency'' or ``persistence''). For instance, a pathogen might encapsulate itself 
to have a higher chance to survive under harsh conditions with the trade-off of no reproduction during the encapsulation. To give a model from this point of view we introduce a Contact Process with infection dormancy (CPID), where only infected individuals exhibit switching. We assume that an infected individual with a dormant infection recovers at a smaller rate than an individual with an active infection, i.e.\ $\delta_d<\delta_a$. In return, only individuals which carry an active infection can infect their neighbours.  Again we assume a spontaneous change from active to dormant and vice versa. The transition graph for this model is given below in Figure \ref{figure:CPID} (left).

To make the CPID comparable to CPS, we consider it artificially as a process with states in $F=\{0,1\}\times\{a,d\}$. One way to do this would be letting $\sigma_0=\rho_0\to\infty$ -- similarly to MTCP. However, there is another interesting way without losing the convenient independence between switching rates and infection, which is given by introducing a \emph{diagonal} arrow from $(0,d)$ to $(1,a)$ (cf.\ Figure~\ref{figure:CPID}, right).

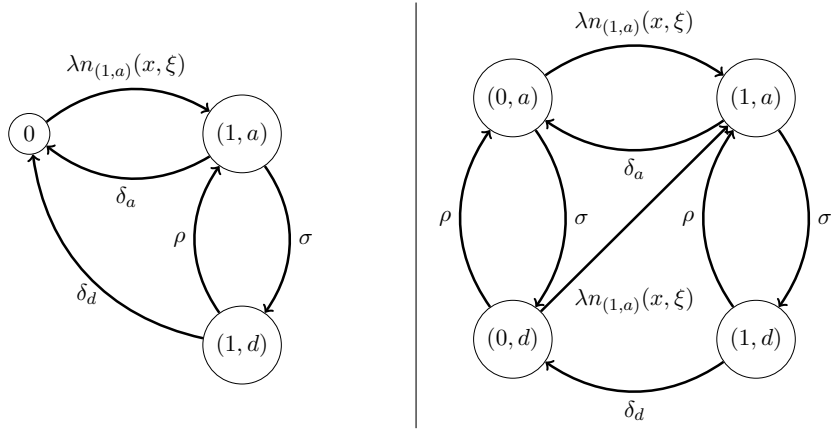
\begin{figure}[h]
	\centering
	\begin{minipage}{0.45\textwidth}
	 \centering
	  \scalebox{0.8}{
		\begin{tikzpicture}
		\node (A) at (0,0) [circle,draw] {$0$};
		\node (B) at (3.5,0) [circle,draw]{$(1,a)$};
		\node (D) at (3.5,-3.5) [circle,draw]{$(1,d)$};
		\draw[->, very thick] (A) to[bend left=35]node[above] {$\lambda n_{(1,a)}(x,\xi)$} (B);
		\draw[<-, very thick] (A) to[bend right=35]node[below] {$\delta_a$} (B);
		\draw[<-, very thick] (A) to[bend right=35]node[below] {$\delta_d$} (D);
		\draw[->, very thick] (B) to[bend left=35]node[right] {$\sigma$} (D);
		\draw[<-, very thick] (B) to[bend right=35]node[left] {$\rho$} (D);
		\end{tikzpicture}}
	\end{minipage}
	\ \vrule\ 
	\begin{minipage}{.45\textwidth}
	  \scalebox{.8}{
		\begin{tikzpicture}
		\node (A) at (0,0) [circle,draw] {$(0,a)$};
		\node (B) at (4,0) [circle,draw]{$(1,a)$};
		\node (C) at (0,-4) [circle,draw]{$(0,d)$};
		\node (D) at (4,-4) [circle,draw]{$(1,d)$};
		\draw[->, very thick] (A) to[bend left=35]node[above] {$\lambda n_{(1,a)}(x,\xi)$} (B);
		\draw[->, very thick] (C) to node[below,yshift=-1cm] {$\lambda n_{(1,a)}(x,\xi)$} (B);
		\draw[<-, very thick] (A) to[bend right=35]node[below] {$\delta_a$} (B);
		\draw[<-, very thick] (C) to[bend right=35]node[below] {$\delta_d$} (D);
		\draw[->, very thick] (A) to[bend left=35]node[right] {$\sigma$} (C);
		\draw[<-, very thick] (A) to[bend right=35]node[left] {$\rho$} (C);
		\draw[->, very thick] (B) to[bend left=35]node[right] {$\sigma$} (D);
		\draw[<-, very thick] (B) to[bend right=35]node[left] {$\rho$} (D);
		\end{tikzpicture}}
	\end{minipage}
	\caption{Left: State transition graph of the CPID.  Right: Equivalent characterization of the CPID on $F^S$.}
	\label{figure:CPID}
\end{figure}
Via such diagonal arrows, infection events directly influence activity states, and this gives rise to additional stochastic dependencies, since activity states do not evolve independently of the infection anymore.
Furthermore, the ability of dormant infections to block infection spread leads to the failure of the coupling arguments that we need to prove either monotonicity, additivity or attractivity, which makes the application of the methods in Sections~\ref{sec:statdis-phasetrans} and \ref{sec:comparison} infeasible.

What one can do, however, is to compare such a CPID $(\cX_t,\cA_t)$ to a CPD $(\cY_t,\cB_t)$ with identical switching and recovery rates.
Then, by comparing the graphical constructions of both processes with identical initial configurations, one sees that $\cA_t\supseteq\cB_t$ for all $t$, which then implies $\cX_t\supseteq\cY_t$.
In turn, the CPID is also dominated by a CP with rates $\lambda$ and $\bar{\delta}$ by the same arguments as in Theorem \ref{thm:upper-bound-CP}.

These couplings at least provide a basic understanding of the behaviour of the CPID. In particular, we can conclude the existence of some parameter choices such that the CPID survives strongly and others such that the process does not survive.
Nevertheless, a more vigorous study of the CPID is desirable.

Lastly, one could also think of taking the diagonal arrow from $(0,a)$ to $(1,d)$ into account, which could correspond to infections exhibiting an incubation period.

\subsection{Open problems and future research}
Obviously, one may consider switching models with more than just two types, and perhaps even infinitely many, opening a wide modelling range.  For the two type case, natural next steps could be the derivation of a complete convergence result, investigation of \emph{strong} survival as well as investigating whether a critical CPS dies out.
Indeed, switching has very recently been introduced into systems of random walks, exclusion and inclusion processes,  in a form that affects jump rates. In this context, after suitable rescaling, it can be shown that `upward flows' invalidating Fick's transport law may emerge, which are not possible in single-type scenarios (\cite{FGHNR22}).\\

\medskip

\bibliographystyle{abbrv}
\bibliography{literature}

\end{document}